\documentclass[12pt]{article}

\usepackage{epstopdf}
\usepackage{latexsym, amsmath, amscd,amsthm}
\usepackage{amssymb}
\usepackage{latexsym, amsmath, amscd,amsthm,booktabs,array,url}
\usepackage{graphicx,epsf}

\newtheorem{theorem}{Theorem}[section]
\newtheorem{corollary}[theorem]{Corollary}
\newtheorem{proposition}[theorem]{Proposition}
\newtheorem{lemma}[theorem]{Lemma}
\newtheorem{conj}{Conjecture}

\theoremstyle{definition}

\newtheorem{example}[theorem]{Example}
\newtheorem{remark}[theorem]{Remark}

\begin{document}

\title{The rank of the walk matrix of the extended Dynkin graph $\tilde{D}_n$}
\author{Sunyo Moon\footnotemark[1]\,  
and Seungkook Park\footnotemark[2]} 

\date{}
\renewcommand{\thefootnote}{\fnsymbol{footnote}}
\footnotetext[1]{Research Institute of Natural Sciences, Hanyang University,Seoul, Republic of Korea(symoon89@hanyang.ac.kr)}
\footnotetext[2]{Department of Mathematics and Research Institute of Natural Sciences, Sookmyung Women's University, Seoul, Republic of Korea(skpark@sookmyung.ac.kr)}
\renewcommand{\thefootnote}{\arabic{footnote}}

\maketitle

\begin{abstract}
In this paper, we provide an explicit formula for the rank of the walk matrix of the extended Dynkin graph $\tilde{D}_n$.
\end{abstract}

\section{Introduction} \label{sec:intro}
Let $G$ be a simple graph with $n$ vertices and let $A$ be an adjacency matrix of $G$. The {\it walk matrix} $W(G)$ of $G$ is defined by
\begin{equation*}
  W(G)=\begin{bmatrix}
         e_n & Ae_n & \cdots & A^{n-1}e_n
       \end{bmatrix},
\end{equation*}
where $e_n$ is the all-one vector of length $n$. The $(i,j)$-entry of the walk matrix $W(G)$ counts the number of walks in $G$ of length $j-1$ from the vertex $i$.
A tree obtained from the path of order $n-1$ by adding a pendant edge at the second vertex is called a {\it Dynkin graph} $D_n$ (See Figure \ref{fig:Dynkin}) and by adding a pendant edge to the second to last vertex of $D_n$, we get the {\it extended Dynkin graph} $\tilde{D}_n$ (See Figure \ref{fig:ExtDynkin}), where $n \geq 4$.

\begin{figure}[h!b!t!]
  \centering
  \includegraphics[width=5cm]{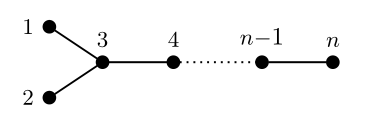}
  \caption{Dynkin graph $D_n$}\label{fig:Dynkin}
\end{figure}

\begin{figure}[h!b!t!]
  \centering
  \includegraphics[width=7cm]{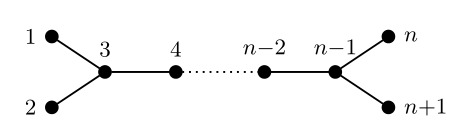}
  \caption{Extended Dynkin graph $\tilde{D}_n$}\label{fig:ExtDynkin}
\end{figure}

Dynkin graphs and extended Dynkin graphs are widely used in the study of simple Lie algebras \cite{Draper, Dynkin, Frappat}, representation theory \cite{Albeverio, Donovan, Gabriel, Kruglyak, Nazarova, Redchuk, Yusenko}, and spectral theory \cite{Cvetkovic, Dokuchaev, Ostrovskii}.
In \cite{Wang22}, Wang, Wang and Guo gave a formula for the rank of the walk matrix of the Dynkin graph $D_n$ ($n \geq 4$). In this paper, we find a formula for the rank of the walk matrix of the extended Dynkin graph $\tilde{D}_n$ ($n \geq 4$) which is $\operatorname{rank}W(\tilde{D}_n)=\lfloor\frac{n}{2}\rfloor$.




For a graph $G$ with $n$ vertices, an eigenvalue of the adjacency matrix of $G$ is said to be {\it main} if its corresponding eigenvector is not orthogonal to $e_n$. In \cite{Hagos}, Hagos showed that the rank of the walk matrix of a graph $G$ is equal to the number of its main eigenvalues.  Thus our formula also provides the number of the main eigenvalues of the walk matrix of the extended Dynkin graph.


\section{Rank of the extended Dynkin graph} \label{sec:result}
Let $\tilde{D}_n$ $(n \geq 4)$ be the extended Dynkin graph with $n+1$ vertices.
We label the vertices of $\tilde{D}_n$ as shown in Figure \ref{fig:ExtDynkin}.
Then the partition $$\Pi=\{\{1,2\},\{3\},\ldots,\{n-1\},\{n,n+1\}\}$$ of $V(\tilde{D}_n)$ is equitable. The $(n+1)\times(n-1)$ characteristic matrix $P$ of $\Pi$ and
 the $(n-1)\times(n-1)$ divisor matrix $B$ of $\Pi$ are
\begin{equation*}
  P=\begin{bmatrix}
    1 & 0 & 0 &  &   \\
    1 & 0 & 0 &  &  \\
    0  & 1 & 0 &  &  \\
     & \ddots & \ddots & \ddots &  \\
       &  & 0 & 1 & 0 \\
       &  & 0 & 0 & 1 \\
       &  & 0 & 0 & 1
  \end{bmatrix}
  ~~\text{and}~~
  B=\begin{bmatrix}
      0 & 1 &  &  &  &  & \\
      2 & 0 & 1 &  &  &  &\\
       & 1 & 0 & 1 &  &  &\\
       & & \ddots & \ddots & \ddots & &\\
       & &  & 1 & 0 & 1 &  \\
       & &  &  & 1 & 0 & 2 \\
       & &  &  &  & 1 & 0
    \end{bmatrix},
\end{equation*}
respectively.

For an $n \times n$ matrix $M$, the {\it Smith normal form} of $M$ is the diagonal matrix $\operatorname{diag}(d_1,d_2,\ldots,d_r,0,\ldots,0)$, where $r=\operatorname{rank}(M)$ and each $d_i$ divides $d_{i+1}$ for $i=1,\ldots,r-1$. Two integral matrices $M_1$ and $M_2$ are {\it integrally equivalent} if there are unimodular matrices $P$ and $Q$ such that $M_2=PM_1Q$, that is, $M_2$ can be obtained from $M_1$ by a finite sequence of elementary row operations and column operations. Two matrices are integrally equivalent if and only if they have the same Smith normal form (See \cite{Gantmacher} for details).

Let $\hat{W}(\tilde{D}_n)$ be the $(n-1)\times(n-1)$ matrix obtained by removing the first and last rows and the last two columns from $W(\tilde{D}_n)$.
Then we define an $(n+1)\times(n+1)$ matrix $W'(\tilde{D}_n)$ in which $\hat{W}(\tilde{D}_n)$ is a submatrix as follows:
\begin{equation*}
    W'(\tilde{D}_n)=\begin{bmatrix}
             0 & 0 & 0 \\
             \hat{W}(\tilde{D}_n) & 0 & 0 \\
             0 & 0 & 0\\
    \end{bmatrix}.
\end{equation*}
For an $n \times n$ matrix $M$, the {\it walk matrix} $W(M)$ of $M$ is defined by
\begin{equation*}
  W(M)=\begin{bmatrix}
         e_n & Me_n & \cdots & M^{n-1}e_n
       \end{bmatrix}.
\end{equation*}

The following lemma provides a method for computing the rank of a walk matrix.

\begin{lemma}\cite[Lemma 2]{Wang22}\label{lem:rank}
  Let $M$ be a real $n\times n$ matrix which is diagonalizable over the real field $\mathbb{R}$. Let $v_1,v_2,\ldots,v_n$ be $n$ linearly independent eigenvectors of $M^T$ corresponding to eigenvalues $\lambda_1,\lambda_2,\ldots,\lambda_n$, respectively. Then we have
  \begin{equation*}
    \det W(M)=\frac{\prod_{1\leq k < j \leq n} (\lambda_j-\lambda_k) \prod_{j=1}^n e_n^T v_j}{\det[v_1,v_2,\ldots,v_n]}.
  \end{equation*}
  Moreover, if $\lambda_1,\lambda_2,\ldots,\lambda_n$ are pairwise different, then $$\operatorname{rank}W(M)=|\{ j:1\leq j \leq n ~~\text{and}~~e_n^Tv_j \neq 0\}|.$$
\end{lemma}

We outline the steps for establishing our formula for the rank of the extended Dynkin graph $\tilde{D}_n$.
\begin{itemize}
    \item [(1)] Using Lemma \ref{lem:rank}, we find a formula for the rank of the walk matrix of $B$.
    \item [(2)] We show that $\hat{W}(\tilde{D}_n)=W(B)$, which implies $\operatorname{rank}\hat{W}(\tilde{D}_n)=\operatorname{rank}W(B)$.
    \item [(3)] We prove that the Smith normal form of $W(\tilde{D}_n)$ and $W'(\tilde{D}_n)$ are identical.
    \item [(4)] We conclude that
    \[
    \operatorname{rank}W(\tilde{D}_n)=\operatorname{rank}W'(\tilde{D}_n)=\operatorname{rank}\hat{W}(\tilde{D}_n)=\operatorname{rank}W(B).
    \]
\end{itemize}


We start by finding the eigenvalues of $B^T$ and the corresponding eigenvectors.

\begin{proposition}\label{prop:eigvec}
Let $\lambda_k=2\cos\frac{k\pi}{n-2}$ for $k=0,\ldots, n-3$ and  $\lambda_{n-2}=-2$. Let
$v_k=\begin{bmatrix}
       1 & \cos\frac{k\pi}{n-2} & \cos\frac{2k\pi}{n-2} & \cdots & \cos\frac{(n-3)k\pi}{n-2}  & \cos{k\pi}
     \end{bmatrix}^T$ for $k=0,\ldots,n-3$
and $v_{n-2}=\begin{bmatrix}
                              1 & -1 & 1& -1&\cdots~
                            \end{bmatrix}^T$.
Then $v_k$ is an eigenvector of $B^T$ corresponding to the eigenvalue $\lambda_k$, that is, $B^Tv_k=\lambda_kv_k$ for $k=0,\ldots,n-2$.
\end{proposition}
\begin{proof}
  Since $v_0=e_{n-1}$ and $\lambda_0=2$, $B^Tv_0=\lambda_0v_0$.
  It is easily check that $B^Tv_{n-2} =\lambda_{n-2}v_{n-2}$.
  Now, we show that $B^Tv_k=\lambda_k v_k$ for $k=1,\ldots, n-3$.
  Let $\theta=k\pi/(n-2)$. Then
  \begin{equation*}
    B^Tv_k=\begin{bmatrix}
      0 & 2 &  &  &  &   \\
      1 & 0 & 1 &  &  &  \\
       & 1 & 0 & 1 &  &  \\
       & & \ddots & \ddots & \ddots & \\
        &  &  & 1 & 0 & 1 \\
        &  &  &  & 2 & 0
    \end{bmatrix}
    \begin{bmatrix}
       1 \\ \cos\theta \\ \cos2\theta \\ \vdots \\ \cos(n-3)\theta  \\ \cos{k\pi}
     \end{bmatrix}
     =
    \begin{bmatrix}
             2\cos \theta \\ 1+\cos2\theta \\ \cos\theta+\cos3\theta \\ \vdots \\ \cos(n-4)\theta+\cos k\pi \\ 2\cos(n-3)\theta
           \end{bmatrix}.
  \end{equation*}
  Since $\lambda_k=2\cos\theta$ and $\cos{m\theta}+\cos{(m+2)\theta}= 2\cos\theta\cos(m+1)\theta$, we obtain $$\cos{m\theta}+\cos{(m+2)\theta}= \lambda_k\cos(m+1)\theta$$ for $m=0,\ldots,n-4$.
  We note that $$2\cos(n-3)\theta =2\big(\cos(n-2)\theta \cos\theta +\sin(n-2)\theta \sin\theta\,\big).$$
  Since $\sin(n-2)\theta=\sin k\pi=0$ for all $k=1,\ldots,n-3$, we have
  \begin{equation*}
    2\cos(n-3)\theta=2\cos\theta\cos(n-2)\theta=\lambda_k \cos k\pi.
  \end{equation*}
  Hence $B^Tv_k=\lambda_kv_k$ for all $k=0,\ldots,n-2$.
\end{proof}

The following lemma will be used in the proof of Lemma \ref{lem:dotprd}.
\begin{lemma}\cite[Lemma 7]{Wang22}\label{lem:sum-cos}
  \begin{equation*}
    \sum_{k=1}^n \cos(ak+b)x=\frac{1}{2\sin{\frac{1}{2}ax}}\bigg( \sin\bigg(\bigg(n+\frac{1}{2}\bigg)a+b\bigg)x-\sin\bigg(\frac{1}{2}a+b\bigg)x \bigg)
  \end{equation*}
  for $\sin{\frac{1}{2}ax}\neq 0$.
\end{lemma}

\begin{lemma}\label{lem:dotprd}
  Let $v_k=\begin{bmatrix}
       1 & \cos\frac{k\pi}{n-2} & \cos\frac{2k\pi}{n-2} & \cdots & \cos\frac{(n-3)k\pi}{n-2}  & \cos{k\pi}
     \end{bmatrix}^T$ for $k=0,\ldots,n-3$
and $v_{n-2}=\begin{bmatrix}
                              1 & -1 & 1& -1&\cdots~
                            \end{bmatrix}^T$.
  Then
  \begin{equation*}
    e_{n-1}^Tv_0=n-1~~~~\text{and}~~~~ e_{n-1}^Tv_k=\begin{cases}
                   1, & \mbox{if $k$ is even}\\
                   0, & \mbox{if $k$ is odd}
                 \end{cases}
  \end{equation*}
  for $k=1,\ldots, n-2$.
\end{lemma}

\begin{proof}
  It is straightforward to check that
  \begin{equation*}
    e_{n-1}^Tv_0=n-1\qquad
  \text{and}\qquad
    e_{n-1}^Tv_{n-2}=\begin{cases}
      1, & \mbox{if $n$ is even}  \\
      0, & \mbox{if $n$ is odd}.
    \end{cases}
  \end{equation*}
  Now, we give the proof for $k=1,\ldots,n-3$.
  Let $\theta=k\pi/(n-2)$. Then
  \begin{equation*}
    e_{n-1}^Tv_k= 1+ \cos k\pi +\sum_{m=1}^{n-3} \cos m\theta.
  \end{equation*}
  By Lemma \ref{lem:sum-cos}, we have
  \begin{align*}
    \sum_{m=1}^{n-3} \cos m\theta=&~\frac{1}{2\sin\frac{1}{2}\theta}\bigg(\sin\bigg(n-\frac{5}{2}\bigg)\theta -\sin \frac{1}{2}\theta\bigg)\\
    =&\frac{1}{\sin{\frac{1}{2}\theta}}\cos\frac{k}{2}\pi \sin\frac{n-3}{2}\theta.
  \end{align*}
  If $k$ is odd, then $\cos k\pi=-1$ and $\cos \frac{k}{2}\pi=0$ and hence $e_{n-1}^Tv_k=0$.
  Suppose that $k$ is even.
  Then $\cos k\pi=1$ and $\cos \frac{k}{2}\pi=\pm1$.
  Since
  \begin{align*}
    \sin \frac{n-3}{2}\theta= & \sin\bigg( \frac{n-2}{2}\theta -\frac{1}{2}\theta \bigg) \\
    = & \sin\frac{n-2}{2}\theta\cos\frac{1}{2}\theta -\sin\frac{1}{2}\theta\cos\frac{n-2}{2}\theta \\
    = & \sin\frac{k}{2}\pi\cos\frac{1}{2}\theta -\sin\frac{1}{2}\theta\cos\frac{k}{2}\pi \\
    = & -\sin\frac{1}{2}\theta \cos\frac{k}{2}\pi,
  \end{align*}
  we obtain
  \begin{equation*}
    \sum_{m=1}^{n-3} \cos m\theta=-\cos^2\frac{k}{2}\pi=-1.
  \end{equation*}
  Hence $e_{n-1}^Tv_k=1$.
\end{proof}
We can now formulate the rank of the walk matrix of $B$.
\begin{theorem}\label{thm:rankB}
  The rank of the walk matrix of $B$ is $\lfloor \frac{n}{2} \rfloor$.
\end{theorem}
\begin{proof}
  Let $v_0,\ldots, v_{n-2}$ be the eigenvectors of $B^T$ in Proposition \ref{prop:eigvec}.
  Then there are $1+\lfloor \frac{n-2}{2} \rfloor=\lfloor\frac{n}{2}\rfloor$ eigenvectors of $B^T$ such that $e_{n-1}v_k \neq 0$ by Lemma \ref{lem:dotprd}.
  We note that all eigenvalues of $B^T$ are pairwise different, since the cosine function is monotonically decreasing on $[0,\pi]$.
  By Lemma \ref{lem:rank}, we have
  \begin{equation*}
      \operatorname{rank}W(B)= \bigg\lfloor\frac{n}{2}\bigg\rfloor.
  \end{equation*}
\end{proof}

\begin{lemma}\label{lem:til=B}
  Let $\tilde{D}_n$ be the extended Dynkin graph with $n+1$ vertices.
  Then $\hat{W}(\tilde{D}_n)=W(B)$.
\end{lemma}
\begin{proof}
  Since $AP=PB$, $A^kP=PB^k$ for all $k \geq 0$.
  Then we have $A^ke_{n+1}=PB^ke_{n-1}$ for all $k \geq 0$, because  $Pe_{n-1}=e_{n+1}$.
   It follows that
  \begin{equation*}
    \begin{bmatrix}
      e_{n+1} & Ae_{n+1} & \cdots & A^{n-2}e_{n+1}
    \end{bmatrix}
    =P
    \begin{bmatrix}
      e_{n-1} & Be_{n-1} & \cdots & B^{n-2}e_{n-1}
    \end{bmatrix}.
  \end{equation*}
  If we delete the first and last rows from $\begin{bmatrix}
      e_{n+1} & Ae_{n+1} & \cdots & A^{n-2}e_{n+1}
    \end{bmatrix}$ and $P$, we get $\hat{W}(\tilde{D}_n)$ and the identity matrix, respectively.
    Hence  $\hat{W}(\tilde{D}_n)=W(B)$.
\end{proof}
\begin{remark}\label{rem:rankhat}
  By Theorem \ref{thm:rankB} and Lemma \ref{lem:til=B}, we have
    \begin{equation*}
      \operatorname{rank}\hat{W}(\tilde{D}_n)=\operatorname{rank}W(B)= \bigg\lfloor\frac{n}{2}\bigg\rfloor.
    \end{equation*}
\end{remark}

%

In the next proposition, we show that $W(\tilde{D}_n)$ and $ W'(\tilde{D}_n)$ are integrally equivalent and hence proving that they have the same Smith normal form.

\begin{proposition}\label{prop:rank}
Let $\tilde{D}_n$ be the extended Dynkin graph with $n+1$ vertices. Then
$W(\tilde{D}_n)$ and $ W'(\tilde{D}_n)$ have the same Smith normal form.
\end{proposition}
\begin{proof}
  Let $r=\operatorname{rank}W(\tilde{D}_n)$.
  Then the first $r$ columns of $W(\tilde{D}_n)$ are linearly independent.
  Note that the first two rows of $W(\tilde{D}_n)$ are the same and the last two rows of $W(\tilde{D}_n)$ are the same.
  Since $r \leq n-1$, the last two columns of $W(\tilde{D}_n)$ can be written as linear combinations of $e_{n+1}$, $Ae_{n+1}$, $\ldots$, $A^{r-1}e_{n+1}$ with integer coefficients.
  Hence we obtain $W'(\tilde{D}_n)$ by using the elementary row and column operations on $W(\tilde{D}_n)$. Therefore, $W(\tilde{D}_n)$ and $ W'(\tilde{D}_n)$ have the same Smith normal form.
\end{proof}

\begin{corollary}
  Let $\tilde{D}_n$ be the extended Dynkin graph with $n+1$ vertices. Then
  $$\operatorname{rank}W(\tilde{D}_n)=\bigg\lfloor\frac{n}{2}\bigg\rfloor.$$
\end{corollary}
\begin{proof}
   By Proposition \ref{prop:rank} and Remark \ref{rem:rankhat},
  \begin{equation*}
    \operatorname{rank}W(\tilde{D}_n) =\operatorname{rank}W'(\tilde{D}_n) =\operatorname{rank}\hat{W}(\tilde{D}_n) =\operatorname{rank}W(B)=\bigg\lfloor\frac{n}{2}\bigg\rfloor.
  \end{equation*}
\end{proof}
We give an example of Proposition \ref{prop:rank}.
\begin{example}\label{ex:snf}
Let $\tilde{D}_8$ be the extended Dynkin graph. Then
\begin{equation*}
    W(\tilde{D}_8)=\left[ \begin {array}{ccccccccc} 1&1&3&4&11&16&43&64&171\\ \noalign{\medskip}1&1&3&4&11&16&43&64&171\\ \noalign{\medskip}1&3&4&11&16&43&64&171&256\\ \noalign{\medskip}1&2&5&8&21&32&85&128&341\\ \noalign{\medskip}1&2&4&10&16&42&64&170&256\\ \noalign{\medskip}1&2&5&8&21&32&85&128&341\\ \noalign{\medskip}1&3&4&11&16&43&64&171&256\\ \noalign{\medskip}1&1&3&4&11&16&43&64&171\\ \noalign{\medskip}1&1&3&4&11&16&43&64&171\end {array} \right]
\end{equation*}
and
\begin{equation*}
    W'(\tilde{D}_8)=\left[ \begin {array}{ccccccccc}
    0&0&0&0&0&0&0&0&0\\
    \noalign{\medskip}1&1&3&4&11&16&43&0&0\\
    \noalign{\medskip}1&3&4&11&16&43&64&0&0\\ \noalign{\medskip}1&2&5&8&21&32&85&0&0\\ \noalign{\medskip}1&2&4&10&16&42&64&0&0\\ \noalign{\medskip}1&2&5&8&21&32&85&0&0\\ \noalign{\medskip}1&3&4&11&16&43&64&0&0\\ \noalign{\medskip}1&1&3&4&11&16&43&0&0\\
    \noalign{\medskip}0&0&0&0&0&0&0&0&0\\\end {array} \right].
\end{equation*}
The Smith normal forms of $W(\tilde{D_8})$ and $W'(\tilde{D_8})$ are $\operatorname{diag}(1,1,1,7,0,0,0,0,0)$ and the rank of $W(\tilde{D}_8)$ is $4$.
\end{example}
Since the number of main eigenvalues of $W(\tilde{D}_n)$ is equal to the rank of $W(\tilde{D}_n)$, we have the following corollary.
\begin{corollary}\label{coro:maineig}
Let $\tilde{D}_n$ be the extended Dynkin graph with $n+1$ vertices.
Then the number of main eigenvalues of $\tilde{D}_n$ is $\lfloor\frac{n}{2}\rfloor$.
\end{corollary}

\section{Future works}
In this paper, we show that the rank of the walk matrix of the extended Dynkin graph $\tilde{D}_n$ is $\lfloor\frac{n}{2}\rfloor$.
The next step would be to find the Smith normal form of the walk matrix of $\tilde{D}_n$. We give the following conjecture.
\begin{conj}
Let $\tilde{D}_n$ be the extended Dynkin graph with $n+1$ vertices. Then the Smith normal form of $W(\tilde{D}_n)$ is
\begin{equation*}
    \begin{cases}
        \operatorname{diag}(\underbrace{1,\ldots,1}_{r-1},n-1,0,\ldots,0), & \mbox{if $n$ is even,}\\
        \operatorname{diag}(\underbrace{1,\ldots,1}_{r-1},\frac{n-1}{2},0,\ldots,0), & \mbox{if $n$ is odd,}
    \end{cases}
\end{equation*}
where $r=\lfloor \frac{n}{2}\rfloor$. 
\end{conj}

\bibliographystyle{plain}

\end{document}